\documentclass{amsart}
\usepackage{amssymb, amsthm, amsmath, amsfonts,amscd,bm,fancyhdr}
\usepackage{graphics}
\usepackage{hyperref}
\usepackage[all]{xy}
\usepackage{enumerate}
\usepackage[mathscr]{eucal}
\usepackage{bbm}
\usepackage{tikz}
\usetikzlibrary{decorations.pathreplacing,angles,quotes,knots}
\usepackage{parskip}

\providecommand{\U}[1]{\protect\rule{.1in}{.1in}}
\def\theenumi{\arabic{enumi}}

\def\theenumii{\alph{enumii}}
\def\p@enumii{\theenumi.}

\def\theenumiii{\arabic{enumiii}}
\def\p@enumiii{(\theenumi)(\theenumii)}

\def\p@enumiv{\p@enumiii.\theenumiii}

\def\cO{\mathcal{O}}

\theoremstyle{plain}
\newtheorem{theorem}{Theorem}[section]

\newtheorem{lemma}[theorem]{Lemma}

\newtheorem{proposition}[theorem]{Proposition}
\newtheorem{prop}[theorem]{Proposition}

\numberwithin{equation}{section}

\theoremstyle{definition}

\newtheorem{definition}[theorem]{Definition}
\newtheorem{example}[theorem]{Example}
\newtheorem{eg}[theorem]{Example}
\newtheorem{remark}[theorem]{Remark}

\newtheorem{thmab}{Theorem}

\newtheorem{corab}[thmab]{Corollary}

\DeclareMathOperator{\FI}{FI}

\DeclareMathOperator{\Hom}{Hom}

\DeclareMathOperator{\Res}{Res}
\DeclareMathOperator{\Ind}{Ind}

\DeclareMathOperator{\Sp}{Sp}

\newcommand{\Sn}{\mathfrak{S}}
\newcommand{\as}{\text{*}}

\newcommand{\R}{\mathbb{R}}

\newcommand{\Q}{\mathbb{Q}}

\newcommand{\dt}{\bullet}
\newcommand{\arXiv}[1]{\href{http://arxiv.org/abs/#1}{\nolinkurl{arXiv:#1}}}
\newcommand{\arXivV}[2]{\href{http://arxiv.org/abs/#1}{\nolinkurl{arXiv:#1v#2}}}

\newcommand{\into}{\hookrightarrow}

\title{$\FI$--sets with relations}

\author[E.~Ramos]{Eric Ramos}
\address[E. Ramos]{University of Michigan Department of Mathematics, 530 Church St., Ann Arbor, MI 48109}
\email{egramos@umich.edu}

\author[D.~E.~Speyer]{David Speyer}
\address[D. Speyer]{University of Michigan Department of Mathematics, 530 Church St., Ann Arbor, MI 48109}
\email{speyer@umich.edu}

\author[G.~White]{Graham White}
\address[G. White]{Indiana University - Bloomington Department of Mathematics, Rawles Hall, Bloomington, IN 47405}
\email{grrwhite@iu.edu}

\thanks{The first author was supported by NSF grant DMS-1704811. The second author was supported by NSF grant DMS-1600223.}

\begin{document}

\begin{abstract}
Let $\FI$ denote the category whose objects are the sets $[n] = \{1,\ldots, n\}$, and whose morphisms are injections. 
We study functors from the category $\FI$ into the category of sets. We write $\Sn_n$ for the symmetric group on $[n]$.
Our first main result is that, if the functor  $[n] \mapsto X_n$ is ``finitely generated" there there is a finite sequence of integers $m_i$ and a finite sequence of subgroups $H_i$ of $\Sn_{m_i}$ such that, for $n$ sufficiently large, $X_n \cong \bigsqcup_i \Sn_n/(H_i \times \Sn_{n-m_i})$ as a set with $\Sn_n$ action. Our second main result is that, if $[n] \mapsto X_n$ and $[n] \mapsto Y_n$ are two such finitely generated functors and $R_n \subset X_n \times Y_n$ is an $\FI$--invariant family of relations, then the $(0,1)$ matrices encoding the relation $R_n$, when written in an appropriate basis, vary polynomially with $n$. In particular, if $R_n$ is an $\FI$--invariant family of relations from $X_n$ to itself, then the eigenvalues of this matrix are algebraic functions of $n$. As an application of this theorem we provide a proof of a result about eigenvalues of adjacency matrices claimed by the first and last author. This result recovers, for instance, that the adjacency matrices of the Kneser graphs have eigenvalues which are algebraic functions of $n$, while also expanding this result to a larger family of graphs.
\end{abstract}


\maketitle

\section{Introduction}

We begin with a specific example of the sort of phenomenon we seek to explain. 
The Kneser graph $KG(n,r)$ has as vertices the $r$-element subsets of $n$ and has an edge between two vertices if and only if the corresponding subsets are disjoint.
Its adjacency matrix is computed in~\cite[Section 9.4]{GR} to have eigenvalues
\[
\lambda_i := (-1)^i\binom{n-r-i}{r-i} \ \mbox{for}\ 0 \leq i \leq r.
\]
 Moreover, each eigenvalue $\lambda_i$ appears with multiplicity 
\[
\binom{n}{i} - \binom{n}{i-1}.
\]
Therefore for each fixed $r$, we observe the following phenomena
\begin{itemize}
\item The total number of distinct eigenvalues is eventually independent of $n$. Specifically, there are eventually exactly $r+1$ such eigenvalues.
\item The eigenvalues each agree with a function which is algebraic over the field $\Q(n)$. Specifically, these are the functions $(-1)^i\binom{n-r-i}{r-i}$ for $0 \leq i \leq r$.
\item The multiplicity of each eigenvalue agrees with a polynomial in $n$. Specifically, these are the polynomials $\binom{n}{i} - \binom{n}{i-1}$.
\end{itemize}
Similar phenomena can be observed in the spectra of the adjacency matrices of Johnson graphs \cite{BCN}, as well as a variety of other examples (see Section \ref{examples}). The main goal of this paper is to provide a uniform framework under which one can deduce the existence of these behaviors. We achieve this using the techniques of representation stability theory and related fields, as appearing in the works of Church, Ellenberg, Farb, Nagpal, Putman, Sam, Snowden, and many others \cite{CEF,CEFN,CF,P,SS}.

Let $\FI$ denote the category whose objects are the sets $[n] = \{1,\ldots, n\}$, and whose morphisms are injections. For any commutative ring $k$, an \emph{$\FI$--module} is a functor from $\FI$ to the category of $k$--modules. \textbf{In this paper, $k$ will be a field of characteristic zero, which is henceforth fixed.} $\FI$--modules were introduced by Church, Ellenberg, and Farb as a single unifying framework for a large collection of seemingly unrelated phenomena from topology, representation theory, and a variety of other subjects \cite{CEF}.

There has been a recent push in the literature to apply the theory of $\FI$--modules to more traditionally combinatorial fields. In \cite{G}, Gadish introduces a theory of $\FI$--posets, functors from $\FI$ to the category of posets. He then applied this framework to prove non-trivial facts about linear subspace arrangements. In \cite{RW}, the first and third authors consider $\FI$--graphs, functors from $\FI$ to the category of graphs. It is proven in that work that such families of graphs display a variety of asymptotic regularities in their enumerative, topological, and algebraic properties.

What is notable about the works mentioned in the previous paragraph is that they follow a common theme: both $\FI$--posets and $\FI$--graphs can be thought of as a pair of an $\FI$--set with a relation. An \emph{$\FI$--set} is a functor $X_\dt$ from $\FI$ to the category of finite sets. A \emph{relation} between $\FI$--sets $X_\dt$ and $Y_\dt$ is any $\FI$--subset of the product $X_\dt \times Y_\dt$. In the case of graphs this relation is the edge relation, while in the case of posets it is the partial ordering. Note that in both of these examples one has $X_\dt = Y_\dt$. Indeed, the applications of our main theorems will mainly focus on this case (see Corollary \ref{maincor}). As a particular application, we provide a proof of Theorem H from \cite{RW} (see Theorem \ref{rwthm}). All this being said, our main theorems will be proven in the context of general $\FI$--set relations.

\begin{eg}
Let $Z_{\dt}$ be the $\FI$--set where $Z_n = [n]$, with the obvious maps. We can linearize this to give an $\FI$--module $k Z_{\dt}$ where $(k Z_{\dt})_n = k^n$. We write $\{ e_i \}_{i \in [n]}$ for the usual basis of $(k Z_{\dt})_n$. As an $S_n$--module, $k^n = \Sp(n) \oplus \Sp(n-1, 1)$, where $\Sp(\lambda)$ is the Specht module. The regularity of this isotypic decomposition as $n$ grows is an example of what is known as representation stability. 

We have a relation on $Z_{\dt}$ given by $\{ (i,j) \in [n]^2 : i \neq j \}$ in degree $n$. We can linearize this relation to give maps $e_i \mapsto \sum_{j \in [n],\ j \neq i} e_j$ from $k Z_n \to k Z_n$. These maps do \textbf{not} give a map of $\FI$--modules so the existing theory of $\FI$--modules does not let us study them. However, for each $n$, this map is a map of symmetric group representations so, by Schur's lemma, it acts by scalars on $\Sp(n)$ and $\Sp(n-1,1)$. Explicitly, these scalars are $n-1$ on $\Sp(n)$ and $-1$ on $\Sp(n-1,1)$. We want to prove that this sort of simple algebraic dependence on $n$ is what happens in general. 

Our Main Theorem \ref{relationthm}, roughly stated, says that for any finitely generated $\FI$--sets $X_{\dt}$ and $Y_{\dt}$ and any relation $R$ between them, the corresponding linear map $k X_n \to k Y_n$ is given by a matrix whose entries depend polynomially on $n$. One of our first goals is therefore to explain in what sense a family of maps $k X_n \to k X_n$, between different vector spaces of different sizes, can be given by a fixed matrix. We achieve this in Section \ref{keydiagram}.
\end{eg}

Before we study relations, we first explore structural properties of general $\FI$--sets. For an $\FI$--set $X_\dt$ we write $X_n$ to denote its evaluation at $n$, while we use \emph{transition map} to  mean one of the maps $X_m \to X_n$ induced by the $\FI$--structure.
We say that an $\FI$--set $X_\dt$ is \emph{finitely generated in degree $\leq d$} if, for all $n \geq d$, the elements of $X_{n+1}$ are all in the image of some transition map from $X_d$. 
We will say that an $\FI$--set is \emph{finitely generated} if it is finitely generated in some degree. 
 Our first goal will be to prove the following structure theorem for $\FI$--sets. For the remainder of this paper we write $\Sn_n$ for the symmetric group on $n$ letters.

\begin{thmab}\label{structurethm}
Let $X_\dt$ denote an $\FI$--set finitely generated in degree $\leq d$. Then there exists a finite collection of integers $m_i \leq d$, and subgroups $H_i \subseteq \Sn_{m_i}$, such that, for $n$ sufficiently large, we have an isomorphism 
\[
X_n \cong \bigsqcup_i \Sn_n / (H_i \times \Sn_{n-m_i})
\]
as sets with an action of $\Sn_n$.
\end{thmab}


\begin{example} \label{TheoremAExample1}
Let $X_n$ be the set of ordered $m$-tuples of distinct elements of $[n]$, with the obvious induced maps. Then $X_n \cong \Sn_n/(\{ e \} \times \Sn_{n-m} )$ as a set with $\Sn$ action, so there is one term in the disjoint union, with $m_i=m$ and $H_i = \{ e \}$.
\end{example}

\begin{example} \label{TheoremAExample2}
Let $X_n$ be the set of unordered $m$-tuples of distinct elements of $[n]$, with the obvious induced maps. Then $X_n \cong \Sn_n/(\Sn_m \times \Sn_{n-m} )$ as a set with $\Sn$ action, so there is one term in the disjoint union, with $m_i=m$ and $H_i = \Sn_m$.
\end{example}

One can of course interpolate between Examples~\ref{TheoremAExample1} and~\ref{TheoremAExample2} by choosing intermediate subgroups of $\Sn_m$, and can also take disjoint unions of this construction for different choices of $m$. Theorem~\ref{structurethm} states that, once $n$ is sufficiently large, all finitely generated $\FI$-sets are built from these operations.

The \emph{linearization} $k X_{\dt}$ is the $\FI$--module where $(k X)_n$ is the free $k$--module on $X_n$ and the maps are defined in the obvious way. For $x \in X_n$, we'll write $e_x$ for the corresponding basis element of $k X_n$.
Let $R_{\dt}$ be a relation between $X_{\dt}$ and $Y_{\dt}$, meaning an $\FI$--subset of $X_{\dt} \times Y_{\dt}$. Then $R$ induces a sequence of maps
\[ r_n : k X_n \to k Y_n \]
by
\[ r_n(e_x) = \sum_{(x,y) \in R_n} e_y. \]
The maps $r_n$ commute with the $\Sn_n$ action, but they do \textbf{not} form a map of $\FI$--modules.


Given a positive integer $n$, a \emph{partition} of $n$ is a tuple of positive integers $\lambda = (\lambda_1,\ldots,\lambda_r)$ such that $\lambda_j \geq \lambda_{j+1}$ and $\sum_j \lambda_j = n$. The irreducible representations of $\Sn_n$ are in bijection with partitions of $n$ in a standard manner, and we write $\Sp(\lambda)$ for the irreducible representation (over $k$) corresponding to the partition $\lambda$. We recall that $\Hom_{\Sn_n}(\Sp(\lambda), \Sp(\lambda)) = k$. If $W$ is an representation of $\Sn_n$, then we write $W_{\lambda} := \Hom_{\Sn_n}(\Sp(\lambda), W)$. So $W \mapsto W_{\lambda}$ is a functor from $\Sn$--representations to vector spaces and we have a canonical isomorphism $W \cong \bigoplus_{|\lambda|=n} W_{\lambda} \otimes \Sp(\lambda)$. The summand $W_{\lambda} \otimes \Sp(\lambda)$ is called the \emph{$\lambda$--isotypic component} of $W$.
We will write $\alpha_{\lambda}$ for the injection $W_{\lambda} \otimes \Sp(\lambda) \to W$ and $\beta_{\lambda}$ for the surjection $W \to W_{\lambda} \otimes \Sp(\lambda)$.



If $\lambda$ is any partition of $m$, and $n \geq m + \lambda_1$, then we set $\lambda[n] = (n - m, \lambda_1,\ldots,\lambda_r)$.
The main result of representation stability is usually stated as follows: Let $M_{\dt}$ be a finitely generated $\FI$--module over $\Q$. Then, for each partition $\lambda$, the multiplicity $\dim (M_n)_{\lambda[n]}$ is independent of $n$ for $n$ sufficiently large \cite{CF}. 
We will want a slightly stronger statement: Roughly, for each partition $\lambda$, there is a finite dimensional vector space $(M_{\dt})_{\lambda}$ such that, for $n$ sufficiently large, we have canonical isomorphisms $(M_n)_{\lambda[n]} \cong (M_{\dt})_{\lambda}$. We will state this rigorously in Section~\ref{defineLambdaComponent}.

Let $X_\dt$ and $Y_\dt$ denote finitely generated $\FI$--sets and $k$ a characteristic zero field. 
Let $R_\dt$ denote a relation between them with associated maps $r_n:k X_n \rightarrow k Y_n$.
Since $r_n$ is a map of $\Sn$--modules, by Schur's lemma, it induces maps $r_{n, \lambda} : (k X_{n})_{\lambda[n]} \to (k Y_{n})_{\lambda[n]}$. 
Using the canonical isomorphisms $(k X_n)_{\lambda[n]} \cong (kX_{\dt})_{\lambda}$ and $(k Y_n)_{\lambda[n]} \cong (kY_{\dt})_{\lambda}$, we obtain a family of maps 
$r_{n, \lambda}: (k X_{\dt})_{\lambda} \longrightarrow (k Y_{\dt})_{\lambda}$

\begin{thmab}\label{relationthm}
In any bases for the vector spaces $(k X_{\dt})_{\lambda}$ and $(k Y_{\dt})_{\lambda}$, the entries of $r_{n,\lambda}$ depend polynomially on $n$.
\end{thmab}

This theorem has a number of more concrete consequences in the case where $X_{\dt} = Y_{\dt}$.
 
\begin{corab}\label{maincor}
Let $X_\dt$ denote a finitely generated $\FI$--set, and let $R_\dt$ denote a self-relation with associated maps $r_n:k X_n \rightarrow k X_n$. Then for $n \gg 0$:
\begin{enumerate}
\item the number $N$ of distinct eigenvalues of $r_n$ is unchanging in $n$;
\item there exists a finite list of functions $f_1$, \dots, $f_N$, each algebraic over the field $\Q(n)$,   such the complete list of distinct eigenvalues of $r_n$ is given by $f_1(n)$, \dots, $f_N(n)$;
\item for each $i$, the function
\[
n \mapsto \text{the algebraic multiplicity of $f_i(n)$ as an eigenvalue of $r_n$}
\]
agrees with a polynomial in $n$.
\end{enumerate}
\end{corab}

\begin{remark}
In \cite[Theorem H]{RW} a version of the above theorem is claimed in the case of $\FI$--graphs and the edge relation. In that work it is said that the theorem would be proven in this paper. In the final section of this work we will explain why \cite[Theorem H]{RW} follows from Corollary \ref{maincor}.

A natural followup question related to the conclusions of Corollary \ref{maincor} is how one can leverage these statements about eigenvalues to say something about the statistics of random walks being performed on $\FI$-sets who have been paired with a transition relation. This is the topic of an upcoming paper of the first two authors.
\end{remark}

\section*{Acknowledgements}
The first author was supported by NSF grant DMS-1704811 and the second author was supported by NSF grant DMS-1600223. We would like to thank Tom Church for his helpful suggestions related to the proof of Theorem \ref{relationthm}.

\section{\texorpdfstring{$\FI$}{FI}-modules and representation stability}

\subsection{$\FI$--modules}
The present work is largely concerned with structures we refer to as $\FI$--sets. One of the primary tools we will use to study these objects, as well as one of the main motivations for considering $\FI$--sets in the first place, are $\FI$--modules.

$\FI$--modules were introduced by Church, Ellenberg, and Farb in their seminal work \cite{CEF}. It was later discovered that this concept arose in a variety of different, sometimes older, contexts such as the twisted commutative algebras of Sam and Snowden \cite{SS} and the study of polynomial functors (see \cite{D,DV} for modern treatments). In this work we will follow the exposition of Church, Ellenberg, and Farb.

\begin{definition}
Let $\FI$ denote the category whose objects are the sets $[n] := \{1,\ldots,n\}$ and whose maps are injections. An \textbf{$\FI$--module over $k$} is a (covariant) functor from $\FI$ to the category of $k$--modules. In this paper, $k$ is always a field of characteristic zero. We will often write $V_n := V([n])$ for the \textbf{degree $n$ piece of $V$}, and $f_\as := V(f)$.
\end{definition}
Since $\Hom_{\FI}([n], [n])$ is the symmetric group $\Sn_n$, each vector space $V_n$ is an $\Sn_n$--representation.


Our next objective will be to specialize to those objects which are finitely generated in the appropriate sense. One should note that the category of $\FI$--modules and natural transformations is abelian, with abelian operations defined point-wise.

\begin{definition}\label{finitelygenerated}
An $\FI$--module $V_{\dt}$ is said to be \textbf{finitely generated in degree $\leq d$} if there is a finite subset of $\bigsqcup_{n = 0}^d V_n$
which is not contained in any proper submodule of $V_{\dt}$.
\end{definition}

We will now define the analogue of free modules for the $\FI$--setting.

\begin{definition}
Let $W$ be a (left) $k[\Sn_m]$--module. Then the \textbf{free $\FI$--module on $W$} is the $\FI$--module $M(W)$ defined by the assignments
\[
M(W)_n := k[\Hom_{\FI}([m],[n])] \otimes_{\Sn_m} W,
\]
where $k[\Hom_{\FI}([m],[n])]$ is the free $k$--module with basis indexed by $\Hom_{\FI}([m],[n])$, viewed as a right $k[\Sn_m]$--module in the obvious way.
\end{definition}

We make the abbreviations $M(m) = M(k \Sn_m)$ (here $k \Sn_m$ is the regular representation of $\Sn_m$) and $M(\lambda) = M(\Sp(\lambda))$.

\begin{remark}\label{yoneda}
It is an easily verifiable fact that there are natural isomorphisms
\[
\Hom_{\FI}(M(W),V) \cong \Hom_{\Sn_m}(W,V_m).
\]
In particular, a map from $M(i)$ into $V_i$ is equivalent to choosing an element of $V_i$, and $\Hom_{\FI}(M(\lambda), V) \cong (V_{|\lambda|})_{\lambda}$.
\end{remark}

\begin{remark}
Because $k$ is a field of characteristic zero, the $\FI$--module $M(W)$ will always be projective. If we were to consider a general commutative ring $k$, then $M(W)$ is projective if and only if $W$ is a projective $k[\Sn_n]$--module. Group algebras of finite groups over characteristic zero fields are semi-simple, so this is automatic in our setting.
\end{remark}

An $\FI$--module $V$ is \textbf{finitely generated} if there exists a collection of non-negative integers, $(d_i)_{i\geq 0}$, all but finitely many equalling zero, and a surjection
\[
\bigoplus_{i \geq 0} M(i)^{d_i} \rightarrow V \rightarrow 0
\]

The following was first proven by Snowden in \cite{S} when $k$ was a field of characteristic 0, and later expanded to more general $k$ by Church, Ellenberg, Farb and Nagpal in \cite{CEF,CEFN}.

\begin{theorem}[Snowden, \cite{S}; Church, Ellenberg, Farb, and Nagpal \cite{CEF,CEFN}]
Let $V$ be a finitely generated $\FI$--module over a Noetherian ring $k$. Then every submodule of $V$ is also finitely generated.
\end{theorem}

The above Noetherian property is arguably the most powerful tool that one has access to when studying finitely generated $\FI$--modules. We will see that it grants us surprisingly brief, although admittedly non-constructive, proofs of certain combinatorial facts. This philosophy can be seen in the context of $\FI$--graphs and $\FI$--posets in \cite{RW} and \cite{G}, respectively.


We take the opportunity to quote:
\begin{proposition}[Church, Ellenberg, and Farb, \cite{CEF}]\label{tensorgen}
If $V$ is finitely generated in degree $\leq d$ and $W$ is finitely generated in degree $\leq e$, then
\begin{enumerate}
\item $V \oplus W$ is generated in degree $\leq \max\{d,e\}$;
\item $V \otimes W$ is generated in degree $\leq d+e$.
\end{enumerate}
\end{proposition}

\subsection{Representation stability} \label{defineLambdaComponent}

We recall the notation from the introduction: $\Sp(\lambda)$ is the Specht module. For an $\Sn_n$--representation $W$ and a partition $\lambda$ of $n$, we put $W_{\lambda} = \Hom_{\Sn_n}(\Sp(\lambda), W)$. For a partition $\lambda=(\lambda_1, \ldots, \lambda_r)$ and an integer $n$ with $n-|\lambda| \geq \lambda_1$, we put $\lambda[n] = (n-|\lambda|, \lambda_1, \lambda_2, \ldots, \lambda_r)$. When $n-|\lambda| < \lambda_1$, we put $\Sp(\lambda[n]) = 0$. 
We also recall that $M(\lambda)$ is shorthand for $M(\Sp(\lambda))$.

Let $V_{\dt}$ be a finitely generated $\FI$--module. Church, Ellenberg, and Farb proved that $V_{\dt}$ exhibits Representation Stability~\cite[Theorem~1.13]{CEF}. This result has three parts, the one of which that is usually cited is Stability of Multiplicities, which states the following: There is a positive integer $N$ and a sequence of nonnegative integers $m_{\lambda}$ indexed by partitions, such that all but finitely many $m_{\lambda}$ are $0$, and
\[ \dim (V_n)_{\lambda[n]} = m_{\lambda} \ \mbox{for}\ n \geq N . \]
We will need a more precise statement which, as we will explain, is also part of Church, Ellenberg, and Farb's result.

Let $\lambda$ be a partition, and let $n$ be large enough that $\lambda[n]$ is defined. 
So by Remark~\ref{yoneda}, we have a natural isomorphism
\[ \Hom_{\FI}(M(\lambda[n]), V) \cong (V_n)_{\lambda[n]} . \]
By the Pieri rule, there is a unique copy of $\Sp(\lambda[n+1])$ inside\footnote{We will often abbreviate induction from $\Sn_m$ to $\Sn_n$ by $\Ind_m^n$, and similarly write $\Res_m^n$ for restriction.} $M(\lambda[n])_{n+1} = \Ind_n^{n+1} \Sp(\lambda[n])$ so there is a unique up to scalar multiple map $M(\lambda[n+1]) \to M(\lambda[n])$. This induces a map $\Hom_{\FI}(M(\lambda[n]), V) \to \Hom_{\FI}(M(\lambda[n+1]), V)$. 
Church, Ellenberg, and Farb's result states that, for $n$ sufficiently large,  the vector spaces on the two sides of this map are of the same dimension. 
We require
\begin{theorem} \label{stabilization}
For $n$ sufficiently large, the map $\Hom_{\FI}(M(\lambda[n]), V) \to \Hom_{\FI}(M(\lambda[n+1]), V)$ is an isomorphism. 
\end{theorem}

\begin{remark}
This result, and our method of proof, is very similar to ideas from Sam and Snowden~\cite{SS}, particularly Section~2.2. We base our argument on Church, Ellenberg, and Farb~\cite{CEF} rather than Sam and Snowden in order to follow our general pattern of using the former's terminology, and because their paper is slightly earlier.
\end{remark}

\begin{proof} Let $\lambda$ and $\mu$ be partitions with $|\lambda| > |\mu|$ and let $m \leq n$ with $m$ and $n$ large enough that $\lambda[m]$ and $\mu[n]$ are defined. 
So $(V_m)_{\lambda[m]} \otimes \Sp(\lambda[m])$ is the $\lambda[m]$--isotypic component of $V_m$, and likewise for $\mu[n]$. 
We claim that, for any transition map $V_m \to V_n$, the composition 
\[ (V_m)_{\lambda[m]} \otimes \Sp(\lambda[m]) \stackrel{\alpha_{\lambda[m]}}{\longrightarrow} V_m \to V_n \stackrel{\beta_{\mu[n]}}{\longrightarrow} (V_n)_{\mu[n]} \otimes \Sp(\mu[n]) \]
is $0$. To see this, note that this map must be $\Sn_m$--equivariant, where we restrict the right hand side to some $\Sn_m \subset \Sn_n$. 
But, by the Pieri rule, $\Sp(\lambda[m])$ does not occur in $\left. \left( \Sp(\mu[n]) \right)\right|_{\Sn_m}$.
The same argument also shows that, if $|\lambda| = |\mu|$ and $\lambda \neq \mu$, then the composite map $ (V_m)_{\lambda[m]} \otimes \Sp(\lambda[m])  \longrightarrow (V_n)_{\mu[n]} \otimes \Sp(\mu[n])$ is $0$. 

Choose any total ordering of the set of partitions such that $|\lambda| < |\mu|$ implies $\lambda < \mu$. Define
\[ V_n^{\geq \lambda} = \bigoplus_{\mu \geq \lambda} V_{\mu[n]} \otimes \Sp(\mu[n]) \subseteq V_n . \]
Then the above argument shows that the transition maps carry $V_m^{\geq \lambda}$ to $V_n^{\geq \lambda}$ for any $m \leq n$. 
So the $V_m^{\geq \lambda}$ form a submodule of $V_{\dt}$, which we denote $V_{\dt}^{\geq \lambda}$. 

By Remark~\ref{yoneda}, we have natural isomorphisms $\Hom(M(\lambda[n]), V) \cong (V_n)_{\lambda[n]} \cong \Hom(M(\lambda[n]), V^{\geq \lambda})$. Also, since $V_{\dt}^{\geq \lambda}$ is a submodule of the finitely generated $\FI$--module $V_{\dt}$, it is finitely generated itself.  So we may, and do, replace $V$ by $V^{\geq \lambda}$. As a result, we may and do assume that $(V_n)_{\mu[n]}=0$ for $\mu < \lambda$.

Now, suppose that $\Hom_{\FI}(M(\lambda[n]), V) \to \Hom_{\FI}(M(\lambda[n+1]), V)$ is not an isomorphism. Since both of these are vector spaces of dimension $m_{\lambda}$, this means that the map is not surjective. So there is some $U \subsetneq (V_{n+1})_{\lambda[n+1]}$ in which the image of $\Hom_{\FI}(M(\lambda[n]), V)$ lies. Tracing through our isomorphisms, this means that all of our transition maps $(V_n)_{\lambda[n]} \to V_{n+1}$, followed by projection onto the $\lambda[n+1]$ isotypic component, land in $U \otimes \Sp(\lambda[n+1])$. Also, for every $\mu \neq \lambda$, the transition maps  $(V_n)_{\mu[n]} \to V_{n+1}$ project to $0$ in the $\lambda[n+1]$ isotypic component by our observations in the first paragraph of the proof. But~\cite[Proposition~3.3.3]{CEF} states that the images of the transition maps from $V_n$ span $V_{n+1}$. This contradiction establishes that $\Hom_{\FI}(M(\lambda[n]), V) \to \Hom_{\FI}(M(\lambda[n+1]), V)$ is an isomorphism after all.
\end{proof}

We define $(V_{\dt})_{\lambda}$ to be the inductive limit $\lim_{n \to \infty} \Hom_{\FI}(M(\lambda[n]), V)$.
So, by Theorem~\ref{stabilization}, $(M_{\dt})_{\lambda} \cong (V_n)_{\lambda[n]}$ for all sufficiently large $n$. 
Thus, if $V_{\dt}$ and $W_{\dt}$ are two $\FI$ modules, and $r_n: V_n \to W_n$ are a sequence of maps of $\Sn_n$--representations then, for $n$ large, each $r_n$ induces a map $r_{n, \lambda}: (V_{\dt})_{\lambda} \to (W_{\dt})_{\lambda}$. 
These are the maps referred to in Theorem~B.

The inclusion $M(\lambda[n+1]) \to M(\lambda[n])$ is only unique up to multiplication by a scalar. We fix choices of these scalars once and for all for the purpose of defining $(V_{\dt})_{\lambda}$. If we changed to other scalars, there would be a canonical isomorphism between the old and the new $(V_{\dt})_{\lambda}$, and the maps $r_{n, \lambda}$ above would be unchanged.

\section{\texorpdfstring{$\FI$}{FI}--sets}

\subsection{Elementary definitions and properties}

In this section we define the fundamental object of study for this paper: $\FI$--sets.

\begin{definition}
An $\FI$--set is a (covariant) functor from $\FI$ to the category of finite sets. We will usually denote $\FI$--sets by $X_\dt$ or $Y_\dt$, where evaluation at $[n]$ is written $X_n$. If $X_\dt$ is an $\FI$--set, and $f:[n] \hookrightarrow [m]$ is an injection, we will generally write  $f_{\ast}$ to denote the induced map, or $X(f)$ if $X$ is not clear from context. For any non-negative integers $n < m$, we will write $\iota_{n,m}:[n] \hookrightarrow [m]$ for the standard inclusion.

We say that $Y_\dt$ is a \textbf{subset} of $X_\dt$ if there is a map of $\FI$--sets $Y_\dt \rightarrow X_\dt$ which is injective point-wise. We say that $X_\dt$ is \textbf{torsion-free} if $X(f)$ is injective for all choices of $f$.
\end{definition}

To the knowledge of the authors, this is the first paper which has formally considered $\FI$--sets. That being said, related structures have been studied in the past. For instance, the first and third authors considered $\FI$--graphs in \cite{RW}, while Gadish studied $\FI$--posets in \cite{G}.

Just as with $\FI$--modules, we will begin by defining finite generation for $\FI$--sets.

%

\begin{definition}\label{fgset}
We say that an $\FI$--set $X_{\dt}$ is \emph{finitely generated in degree $\geq d$} if there is a finite subset of $\bigsqcup_{i=0}^d X_i$ which is not contained in any proper $\FI$--subset of $X_{\dt}$. 
We say that $X_{\dt}$ is finitely generated if it is finitely generated in some degree.
\end{definition}
 Note, if $X_{\dt}$ is finitely generated, then all the $X_n$ are finite.

\begin{definition}
For any $\FI$--set $X_{\dt}$, the \emph{linearization} $k X_{\dt}$ is the $\FI$--module where $(kX_{\dt})_n$ is the free $k$--vector space on the set $X_n$, with the obvious maps. 
\end{definition}

Linearization is a functor from $\FI$--sets to $\FI$--modules.
We see that $X_\dt$ is finitely generated if and only if $k X_\dt$ is (and in the same degree). This yields some immediate consequences:

\begin{prop} \label{elementstab}
Let $X_\dt$ be a finitely generated $\FI$--set. For $n$ sufficiently large, and  $f:[n] \hookrightarrow [n+1]$ any injection, the map $f_{\ast}: X_n \to X_{n+1}$ is injective.
\end{prop}

\begin{proof}
The map $f_{\ast}: X_n \to X_{n+1}$ is injective if and only if the linearization  $f_{\ast}: kX_n \to kX_{n+1}$ is, and the latter is true for $n \gg 0$ by the representation stability theorem of~\cite{CEF}.
\end{proof}

\begin{prop} \label{numberorbitsstab}
Let $X_\dt$ be a finitely generated $\FI$--set. For $n$ sufficiently large, the number of orbits of $\Sn_n$ on $X_n$ is a constant independent of $n$. 
\end{prop}

\begin{proof}
The number of orbits of $\Sn_n$ on $X_n$ is the multiplicity of $\Sp(n)$ in $k X_n$. Since $\Sp(n) = \Sp(\emptyset[n])$, this is eventually constant by~\cite{CEF}.
\end{proof}

Every inclusion $f: [n] \to [n+1]$ induces a map on the orbit sets $X_n/\Sn_n \to X_{n+1}/\Sn_{n+1}$, and all of these maps are the same.

\begin{prop} \label{EventuallyBijective}
Let $X_\dt$ be a finitely generated $\FI$--set. For $n$ sufficiently large, the map $X_n/\Sn_n \to X_{n+1}/\Sn_{n+1}$ described above is bijective.
\end{prop}

\begin{proof}
By Proposition~\ref{numberorbitsstab}, for $n$ large enough, $|X_n/\Sn_n|  = |X_{n+1}/\Sn_{n+1}|$. So, it is enough to show that the map is surjective for $n$ sufficiently large. Suppose that $u \in X_{n+1}$ is such that the orbit $\Sn_{n+1} u$ is not in the image of this map. Then the corresponding basis vector $e_u$ of $k X_{n+1}$ is not in $\mathrm{Span}_{\sigma \in \Sn_{n+1}} \sigma k X_n$. But finite generation implies that $\mathrm{Span}_{\sigma \in \Sn_{n+1}} \sigma V_n=V_{n+1}$ for $n$ sufficiently large, in any finitely generated $\FI$--module $V_{\dt}$.
\end{proof}

\begin{definition}
The previous proposition implies that we may define,
\[
X_{\dt}/\Sn := \lim_{n \to \infty} X_n/\Sn_n.
\]
Elements of $X_{\dt}/\Sn$ will be referred to as the \textbf{stable orbits}, or just the \textbf{orbits} of $X_\dt$.
\end{definition}

\subsection{Induced $\FI$--sets}

In this section we discuss the properties of what we call induced $\FI$--sets.

\begin{definition}\label{inducedset}
Let $m$ be a fixed non-negative integer, and let $X$ be an $\Sn_m$--set. 

For any $n \geq m$, we define $M(X)_n$ to be the set of ordered pairs $(f,x)$
where $x \in X$ and $f \in \Hom_{\FI}([m],[n])$ is strictly increasing. For $n < m$, we set $M(X)_n = \emptyset$.
Given an injection $g:[n] \to [p]$, we define $g_{\ast} : M(X)_n \to M(X)_p$ as follows: 
We can uniquely factor $g \circ f$ as $h \circ \sigma$ where $\sigma \in \Sn_m$ and $h$ is strictly increasing. We put
\[ g_{\ast}(f,x) = (h,\sigma x) .\]

$\FI$-sets of the form $M(X)_\dt$ are known as \textbf{induced $\FI$-sets}.\\
\end{definition}

The following lemma is clear from the definition of the induced $\FI$--set, and of induced representations. 
\begin{lemma}
Let $X$ be an $\Sn_m$--set. Then 
\[
k M(X_{\dt}) \cong M(k X_{\dt}).
\]
\end{lemma}

We observe that Theorem \ref{structurethm} is straight-forward for induced $\FI$--sets.

\begin{lemma}\label{inducedversion}
Let $X$ be an $\Sn_m$--set, and write
\[
X = \bigsqcup_{i \in I} \Sn_m/H_i
\]
for some indexing set $I$. Then for any $n \geq 0$,
\[
M(X)_n = \bigsqcup_{i \in I} \Sn_n/(H_i \times \Sn_{n-m}),
\]
where implicitly $\Sn_n/(H_i \times \Sn_{n-m})$ is empty for $n < m$.
\end{lemma}

\begin{proof} 
We must show that the stabilizers of $M(X)_n$ are of the form $H_i \times \Sn_{n-m}$. Indeed, given an element $(f,x)$, where $x$ is stabilized by $H_i$, for any $g \in \Sn_n$, we have $g_{\ast}(f,x) = (f,x)$ if and only if $g \circ f = f \circ \sigma$ for $\sigma \in H_i$.
This will occur if and only if $g$ maps $[m]$ to itself by $\sigma$. So $g$ stabilizes $(f,x)$ if and only if $g \in H_i \times \Sn_{n-m}$.
\end{proof}

One way to interpret Theorem~\ref{structurethm} is that every finitely generated $\FI$--set eventually ``looks like'' an induced $\FI$--set. More precisely, if $X_\dt$ is a finitely generated $\FI$--set which is generated in degree $\leq d$, then there exists a collection $\{Y_i\}_{i=0}^d$, with $Y_i$ a $\Sn_i$--set, and an isomorphism of $\Sn_n$--sets
\[
X_n \cong \sqcup_i M(Y_i)_n,
\]
for all $n \gg 0$. An analogous theorem has been known to be true about $\FI$--modules since at least the work of Nagpal \cite{N}.\\

\subsection{Motivating examples} \label{examples}

In this section, we take time to write down a collection of motivating examples for the study of $\FI$--sets. Our focus will be on constructing illustrative examples of $\FI$--graphs. An $\FI$--graph is a functor from $\FI$ to the category of graphs. In other words, an $\FI$--graph is an $\FI$--set of vertices, paired with a symmetric relation which dictates how these vertices are connected through edges (see \cite{RW}).

Our study of $\FI$--graphs begins with Kneser graphs.

\begin{example}\label{ex:kneser}
For any fixed $n,r \geq 0$, the Kneser graph $KG_{n,r}$ has vertices indexed by the $r$--element subsets of $[n]$, with edges between two vertices if those subsets are disjoint. 

The $\FI$--graph $KG_{\dt,r}$ has $G_n = KG_{n,r}$. For each injection $f$ from $[m]$ to $[n]$, the corresponding transition map takes the vertex $\{a_1,\dots,a_r\}$ of $G_m$ to the vertex $\{f(a_1),\dots,f(a_r)\}$ of $G_n$.
\end{example}

There are several minor ways in which this construction can be generalized, as in the following examples. 

The vertices could be indexed by ordered $r$--tuples rather than by (unordered) subsets.

\begin{example}\label{ex:kneserordered}
For any fixed $r \geq 0$, define each graph $G_{n}$ to have vertices indexed by the $r$--tuples of elements of $[n]$, with edges between two vertices if no element of $[n]$ appears in both $r$--tuples. For each injection $f$ from $[n]$ to $[m]$, the corresponding transition map takes the vertex $(a_1,\dots,a_r)$ of $G_n$ to the vertex $(f(a_1),\dots,f(a_r))$ of $G_m$.

As with Example \ref{ex:kneser}, these graphs and transition maps form an $\FI$--graph. 
\end{example}

Future examples with the same vertex sets as Example \ref{ex:kneser} or \ref{ex:kneserordered} will use the same transition maps, without this being explicitly stated each time.

Rather than using ordered or unordered $r$--tuples, as in Examples \ref{ex:kneser} and \ref{ex:kneserordered}, it is possible to care about only some of the order data.

\begin{example}\label{ex:kneserpartordered}
For any fixed $r \geq 0$ and subgroup $H$ of the symmetric group $\Sn_r$, define each graph $G_{n}$ to have vertices indexed by orbits of $r$--tuples of elements of $[n]$ under the action of $H$. As in Examples \ref{ex:kneser} and \ref{ex:kneserordered}, edges are placed between each pair of vertices labelled by disjoint $r$--tuples.
\end{example}

Other $\FI$--graphs may be defined with the same vertex sets as in Examples \ref{ex:kneser}, \ref{ex:kneserordered}, or \ref{ex:kneserpartordered}, but with different sets of edges. There may be multiple orbits of edges, and they may only exist from a certain degree onwards.

\begin{example}\label{ex:kneseredgeorbits}
Let the vertex set of $G_n$ be indexed by $r$--element subsets of $[n]$, and let $a_0$ to $a_r$ be positive integers or infinity. In $G_n$, there is an edge between two vertices which share exactly $l$ elements if and only if $n \geq a_l$. 
\end{example}

This example has $r+1$ orbits of pairs of vertices, determined by the size $i$ of the intersection of the two labelling sets. For each of these orbits, there is an edge joining those two vertices from degree $a_i$ onwards.

If the vertices are described by (ordered) $r$--tuples as in Example \ref{ex:kneserordered}, then there are many more orbits of pairs of vertices --- rather than these orbits being defined just by the size of the intersection, they also take into account which positions any equal entries occupy. As in Example \ref{ex:kneseredgeorbits}, though, all that is required is to choose when edges appear in each vertex orbit. In this example, the orbits of pairs of vertices are a little more complicated.

\begin{example}\label{ex:kneserorderededgeorbits}
Let the vertex set of $G_n$ be indexed by $r$--tuples of elements of $[n]$. For each integer $l$ between $0$ and $r$ and each injection $s$ from any $l$--element subset of $[r]$ to $[r]$, fix $a_{ls}$ to be either a nonnegative integer or infinity. Because we are working with undirected graphs, we require that $a_{ls^{-1}} = a_{ls}$.

In $G_n$, there is an edge between two vertices whose labelling $r$--tuples have $l$ entries in common in positions given by $s$ exactly if $n \geq a_{ls}$. 
\end{example}

The number of parameters $a_i$ or $a_{ls}$ required by Examples \ref{ex:kneseredgeorbits} and \ref{ex:kneserorderededgeorbits} is the number of orbits of pairs of vertices, in the sense of the minimal number of pairs of vertices required for any pair of vertices in any degree to be in their image under some transition map. Effectively, for each orbit $i$ of pairs of vertices, we need to decide in which degree $a_i$ pairs of vertices in this orbit are first connected by an edge. Once this happens, all other pairs of vertices in the same orbit in the graph $G_{a_i}$ must also be connected by an edge, and likewise pairs of vertices in the image of these pairs in later graphs $G_r$, for $r > a_i$. 

Disjoint sums of $\FI$--sets are $\FI$--sets, so the preceding examples may be combined to give larger ones. Such a construction will have additional orbits of pairs of vertices, allowing additional edges as in the following example.

\begin{example}\label{ex:twoknesers}
Choose nonnegative integers $r, l, a_{11}, a_{12},$ and $a_{22}$. Each graph $G_n$ has a vertex for each subset of $[n]$ of size $r$ and another vertex for each subset of size $l$. Color these vertices red and blue, respectively. There is an edge between
\begin{itemize}
\item two red vertices if their subsets have intersection of size $a_{11}$
\item a red vertex and a blue vertex if their subsets have intersection of size $a_{12}$
\item two blue vertices if their subsets have intersection of size $a_{22}$
\end{itemize} 
\end{example}

Example \ref{ex:twoknesers} is the disjoint sum of two instances of Example \ref{ex:kneseredgeorbits}, with additional edges added between red vertices and blue vertices. A more general example could be constructed with more parameters, both those used in Example \ref{ex:kneseredgeorbits}, and new parameters for each orbit of pairs of vertices with one red and one blue.

Because the conditions on an $\FI$--graph only involve maps from $G_n$ to $G_m$ with $m \geq n$, an $\FI$--graph may be edited by removing all vertices and edges before a certain point. 

\begin{example}\label{ex:zeromod}
Let $G_\dt$ be an $\FI$--graph. Modify it by replacing each $G_i$ by the empty graph, for $i = 0$ to $k-1$. Transition maps from these graphs are trivial. The resulting object is an $\FI$--graph.
\end{example}

It is not possible to remove all vertices and edges from any graph after a nonempty graph $G_n$, because transition maps to the empty graph cannot be defined. The closest we can come is to crush the entire graph to a point, perhaps with a self-edge. 

\begin{example}\label{ex:onemod}
Let $G_\dt$ be an $\FI$--graph. Modify it by replacing each $G_i$ by a single vertex, for $i \geq k$. If there are any edges in any prior graph, then this vertex must have a self-edge. Transition maps to these single-vertex graphs map every vertex to the only vertex. The resulting object is an $\FI$--graph.
\end{example}

Our desire to in general allow non-injective behavior of the sort described in Example \ref{ex:onemod} is why we allow graphs to have self-edges. If self-edges are forbidden, then this example is only allowed when there are no edges earlier in the $\FI$--graph and in general vertices would only be able to map to the same vertex if they were not connected by an edge.

\subsection{Proof of Theorem~\ref{structurethm}}

We now prove Theorem~\ref{structurethm}. Let $X_{\bullet}$ be a finitely generated $\FI$--set. 
For each $\cO$ in $X_{\bullet}/\Sn$, we get an $\FI$--subset $X(\cO)$ of $X_{\bullet}$ corresponding to the elements which map to $\cO$ under the maps $X_n \longrightarrow X_n/\Sn_n \longrightarrow X_{\bullet}/\Sn$, and we have $X_{\bullet} = \bigsqcup_{\cO} X(\cO)$. 
So it is enough to prove the theorem for each $X(\cO)$. In other words, we may, and do, reduce to the case that $X_{\bullet}/\Sn$ is a singleton.

By Proposition~\ref{EventuallyBijective}, for $n$ large enough, the maps $X_n / \Sn_n \to X_{n+1} / \Sn_{n+1}$ are bijective. 
So, for $n$ large enough, the action of $\Sn_n$ on $X_n$ is transitive.

Choose some $k$ large enough for the action of $\Sn_n$ on $X_n$ to be transitive for all $n \geq k$. Choose some particular element $x \in X_k$. Let $G_k$ be the stabilizer of $x$ in the $\Sn_k$ action on $X_k$.
For all $n \geq k$, let $\iota_{kn}$ be the obvious inclusion of $[k]$ into $[n]$ and let $G_n$ be the stabilizer of $\iota_{kn}(x)$ in $\Sn_n$. 
We want to show that there is a nonnegative integer $m$ and a subgroup $H \subseteq \Sn_m$ such that, for $n$ sufficiently large, the subgroup $G_n$ of $\Sn_n$ is conjugate to $H \times \Sn_{n-m}$.

\begin{lemma} \label{StabilizerGrowth}
For all $n \geq \ell \geq k$, we have $G_{\ell} \times \Sn_{n-\ell} \subseteq G_n$.
\end{lemma}

\begin{proof}
For any $\sigma \in \Sn_{\ell}$ and $\tau \in \Sn_{n-\ell}$,  we have $(\sigma \times \tau) \circ \iota_{\ell n} = \iota_{\ell n} \circ \sigma$.
Now, suppose $\sigma \in G_m$ so $\sigma (\iota_{k \ell}(x)) = \iota_{k \ell}(x)$. Then 
\[ (\sigma \times \tau) \left( \iota_{kn}(x) \right) = (\sigma \times \tau) \circ \iota_{\ell n} \circ \iota_{k\ell}(x) = \iota_{\ell n} \circ \sigma \left(  \iota_{k\ell}(x) \right) = \iota_{\ell n} \circ \iota_{k\ell}(x) = \iota_{kn}(x). \]
So $\sigma \times \tau$ stabilizes $\iota_{kn}(x)$, and thus lies in $G_n$.
\end{proof}

Define $A_n \subseteq [n]$ to be the orbit of $n$ under $G_n$. 

\begin{lemma} \label{AGrows}
For $n > k$, we have $A_{n+1} \supseteq A_n \cup \{ n+1 \}$.
\end{lemma}

We remark that the statement is meaningful for $k=n$ but need not be true in that case.

\begin{proof}
By definition, $n+1 \in A_{n+1}$. So the task is to show that $A_n \subset A_{n+1}$.

By Lemma~\ref{StabilizerGrowth}, $G_{n+1}$ contains $G_{n-2} \times \Sn_2$ and, in particular, contains the transposition $(n \ n+1)$. So $n+1$ and $n$ are in the same $\Sn_{n+1}$ orbit and $n \in A_{n+1}$. But also by Lemma~\ref{StabilizerGrowth}, $G_{n+1}$ contains $G_n \times \{ e \}$. So the $G_{n+1}$ orbit of $n$ contains the $G_n$ orbit of $n$. In other words, $A_n \subset A_{n+1}$ as required.
\end{proof}

Let $B_n = [n] \setminus A_n$. Then Lemma~\ref{AGrows} shows that $[k] \supseteq B_{k+1} \supseteq B_{k+2} \supseteq B_{k+3} \supseteq \cdots$.
For $n$ large enough, therefore, the subset $B_n$ stabilizes at some subset $B$ of $[k]$. Let $m = |B|$. 

For a subset $P$ of $[n]$, let $\Sn_P$ be the subgroup of $\Sn_n$ which fixes all elements of $[n] \setminus P$.
Here is our final, key, lemma:

\begin{lemma} \label{KeyLemma}
Let $n$ be large enough that $2(n-k) > (n-m)$ and $|A_n| = n-m$. Then $\Sn_{A_n} \subseteq G_n$.
\end{lemma}

\begin{proof}
Let $a \in A_n$. It is enough to show that $G_n$ contains the transposition $(a\ n)$.

Let $[k+1, n] = \{ k+1, k+2, \ldots, n \}$. By Lemma~\ref{StabilizerGrowth}, we have $\Sn_{[k+1, n]} \subseteq \Sn_n$.
Since $a \in A_n$, there is some element $\rho \in G_n$ mapping $n$ to $a$. Then $\rho \Sn_{[k+1, n ]} \rho^{-1} = \Sn_{\rho([k+1, n])}$ is in $G_n$ as well.
We have $[k+1, n] \subseteq A_n$, so $\rho([k+1, n]) \subseteq A_n$ and, since $2(n-k) > n-m = |A_n|$, the sets $[k+1, n]$ and $\rho([k+1, n])$ must overlap. So there is some $b \in [k+1, n] \cap \rho([k+1, n])$. Then the transpositions $(b \ n)$ and $(a \ b)$ are in $G_n$, so the transposition $(a\ b) (b \ n) (a\ b) = (a \ n)$ is as well.
\end{proof}

We are now ready to finish the proof. For $n$ large enough that Lemma~\ref{KeyLemma} holds, we know that $\Sn_{A_n} \subseteq G_n \subseteq \Sn_{A_n} \times \Sn_{B}$. This means that $G_n$ must be of the form $\Sn_{A_n} \times H_n$ for some subgroup $H^{(n)}$ of $\Sn_B$. Moreover, by Lemma~\ref{KeyLemma}, we have $H_n \subseteq H_{n+1} \subseteq H_{n+2} \subseteq \cdots \subseteq \Sn_B \cong \Sn_m$ so, for $n$ large enough, the subgroup $H_n$ stabilizes. We take $H$ to be this stable limit. \qedsymbol

\begin{remark}
\label{rem:stabgroupbehavior}
The results of this section actually constrain the behavior of the stabilizer groups $G_n$ quite severely, even when $n$ is not yet `large enough'. As we move from $X_n$ to $X_{n+1}$, the groups $G_n$ and $G_{n+1}$ are related in one of the following ways.
\begin{itemize}
\item It may be that $\Sn_{A_{n+1}} \subseteq G_{n+1}$, in which case the subgroup $H^{(n+1)}$ of $\Sn_{B_{n+1}}$ contains the intersection $H^{(n)} \cap \Sn_{B_{n+1}}$, bearing in mind that $B_{n+1}$ may be smaller than $B_n$.
\item Alternatively, it is possible that $G_{n+1}$ does not contain $\Sn_{A_{n+1}}$. This can only happen when the hypotheses of Lemma \ref{KeyLemma} are not yet satisfied. If this happens, then $\Sn_{A_{n+1} \cup \{n+2\}}$ is contained in $G_{n+2}$. Example \ref{ex:wreath} gives an example of this behavior.
\end{itemize}
\end{remark} 

The second case of Remark \ref{rem:stabgroupbehavior} is why Lemma \ref{KeyLemma} requires that $2(n-k) > (n-m)$. The following example illustrates what may happen when $n$ is not yet this large.

\begin{example}
\label{ex:wreath}
Let $X_n$ be empty for $n < 5$. Take the groups $G_n$ for $n \geq 5$ to be
\begin{itemize}
\item Generated by the symmetric groups $\Sn_{2}$ acting on $[2]$, $\Sn_3$ acting on $\{3,4,5\}$ and $S_{n-5}$ acting on $[n] \backslash [5]$ for $n \in \{5,6,7\}$
\item Generated by $\Sn_{2}$ acting on $[2]$, $\Sn_3$ acting on $\{3,4,5\}$, $\Sn_{3}$ acting on $\{6,7,8\}$, and the permutation $(3 \: 6)(4 \: 7)(5 \: 8)$ for $n = 8$.
\item Generated by $\Sn_{2}$ acting on $[2]$ and $\Sn_{n-2}$ acting on $[n] \backslash [2]$ for $n > 8$
\end{itemize}
Observe the failure of Lemma \ref{KeyLemma} for $G_8$. The orbit $A_8$ is $\{3,4,5,6,7,8\}$, but not all of $\Sn_{A_8}$ is contained in $G_8$. We do not give a complete construction of an $\FI$--set with these stabilizer groups --- the vertices may be taken to be appropriate cosets of the groups $G_n$.
\end{example}

The gist of Remark \ref{rem:stabgroupbehavior} and Example \ref{ex:wreath} is that to go from $G_n$ to $G_{n+1}$, one may remove elements from $B_n$ or increase the subgroup $H^{(n)}$. When $n$ is small, it is also possible to have a wreath product factor appear in $G_{n+1}$. This factor is temporary, in that it will always further increase to a large symmetric group in $G_{n+2}$.

\section{Relations of \texorpdfstring{$\FI$}{FI}--sets}

\subsection{Elementary definitions and properties}

In this section we turn our attention to relations defined by $\FI$--sets.  We recall the definition from the introduction:

\begin{definition}
Let $X_\dt$ and $Y_\dt$ denote two $\FI$--sets. Then a \textbf{relation between $X_\dt$ and $Y_\dt$} is an $\FI$--subset of the product $(X \times Y)_\dt$.  
\end{definition}

In Section~\ref{examples} we examined a large collection of examples of $\FI$--sets and relations.

\begin{proposition}
Let $X_\dt$, $Y_\dt$ be finitely generated $\FI$--sets, and let $R_\dt$ be a relation between $X_\dt$ and $Y_\dt$. Then $R_\dt$ is finitely generated.
\end{proposition}

\begin{proof}
It suffices to prove that the linearization $k R_\dt$ is finitely generated (for any choice of $k$). It is easily seen that
\[
k(X \times Y)_\dt \cong kX_\dt \otimes kY_\dt.
\]
Proposition~\ref{tensorgen} implies that $kX_\dt \otimes kY_\dt$ is finitely generated, whence the same is true of $k(X\times Y)_\dt$. The proposition now follows from the Noetherian property.
\end{proof}

%

\begin{example}
Let $G_\dt$ denote an $\FI$--graph (see \cite{RW}). Then the $\FI$--set encoding the edges of $G_\dt$, $E(G_\dt)$, can be viewed as a symmetric relation between the vertex $\FI$--set and itself. In fact, understanding properties of $E(G_\dt)$ is one of the main motivations of the present work. Theorem \ref{relationthm} can be seen as a vast generalization of Theorem H of \cite{RW} (see Section \ref{mainapp}). Section~\ref{examples} focused on giving a large collection of examples of specific $\FI$--graphs.
\end{example}

\begin{example}
Let $P_\dt$ denote an $\FI$--poset with partial orderings $\leq_\dt$ (see \cite{G}). Then one has an $\FI$--relation defined by
\[
R_n = \{(x,y) \mid x \leq_n y\}.
\]
$\FI$--posets were used by Gadish in \cite{G}, where he showed that they have a variety of applications in studying representation stability phenomena arising from linear arrangements.
\end{example}

Given a relation $R_\dt$ between two $\FI$--sets $X_\dt$ and $Y_\dt$ one may associated a collection of maps $r_n: \ kX_n \rightarrow kY_n$ for any choice of $k$. Namely, for any $x \in X_n$,
\[
r_n(e_x) = \sum_{(x,y) \in R_n} e_y.
\]
Critically, the collection $\{r_n\}_n$ does not necessarily extend to a morphism of $\FI$--modules $k[X_\dt] \rightarrow k[Y_\dt]$. Despite this fact, we want to prove the maps $r_n$  display a regularity as $n$ varies.

\begin{example}
Once again let $G_\dt$ be an $\FI$--graph, and assume that the  vertex sets of $G_\dt$ are finitely generated. If we chose our relation to be the edge relation, then the associated maps $r_n: \Q V(G_n)  \rightarrow \Q V(G_n)$ are the adjacency matrices of the associated graphs. These maps are studied in~\cite{RW}, where it is pointed out that they usually do not form a morphism of $\FI$--modules.
\end{example}

\begin{example}
If we assume that $P_\dt$ is an $\FI$--poset, then the associated maps $r_n$ are sometimes called the incidence matrix of the poset $P_n$. These are the matrices with rows and columns indexed by elements of $P_n$ which have a $1$ in position $(x,y)$ whenever $x \leq y$, and a $0$ otherwise. Note that unlike in the previous case, this matrix is not (necessarily) symmetric. The inverse of $r_n$ is the M\"obius function of the poset $r_n$.
\end{example}

\subsection{A key diagram}\label{keydiagram}

In this section, we begin to detail the main construction used in the proof of Theorem \ref{relationthm}. This construction does not make use of the $\FI$--set structure in its early stages, and can be accomplished at the level of $\FI$--modules. In the next section, we will specialize to the $\FI$--set case, and complete the proof of Theorem \ref{relationthm}

Let $V_{\bullet}$ be a finitely generated $\FI$--module and $\lambda$ a partition. As explained in Section~\ref{defineLambdaComponent}, we define $V_{\lambda}$ to be $\lim_{n \to \infty} \Hom_{\FI}(M(\lambda[n]), V_{\bullet})$. As we showed there, for $n$ sufficiently large, the maps in this inductive limit are isomorphisms, so $V_{\lambda}$ is canonically isomorphic to $(V_n)_{\lambda[n]}$ for any sufficiently large $n$. Any $n$ which is sufficiently large for this purpose will be said to be \textbf{in the stable range}. We recall that the definition of $V_{\lambda}$ required fixing once and for all embeddings $\Sp(\lambda[n]) \into \Ind_m^{n} \Sp(\lambda[m])$; we will use those same embeddings throughout this section.

By Frobenius reciprocity, the inclusion $\Sp(\lambda[n]) \into \Ind_m^{n} \Sp(\lambda[m])$ corresponds to an inclusion $\Sp(\lambda[m]) \into \Res^n_m \Sp(\lambda[n])$. We'll denote this inclusion $\eta_{m,n}$. 

As in Section~\ref{defineLambdaComponent}, we define $V^{\geq \lambda}_n$ to be the subrepresentation of $V_n$  spanned by the $\mu[n]$--isotypic pieces, where $|\mu| \geq |\lambda|$. We define $V^{> \lambda}_n$ to be the subrepresentation of $V^{\geq \lambda}_n$ spanned by the $\mu[n]$--isotypic components with $\mu \neq \lambda$.  As observed in Section~\ref{defineLambdaComponent}, the vector spaces $V^{\geq \lambda}_n$ form a sub-$\FI$--module of $V_{\bullet}$, and the $V^{>\lambda}$ form a sub-$\FI$--module of those. 

Let $\iota_{m,n}: [m] \to [n]$ be the standard inclusion $r \mapsto r$. Then we have transition maps $(\iota_{n,m})_{\ast}: V^{\geq \lambda}_m \to V^{\geq \lambda}_n$ and  $(\iota_{n,m})_{\ast}: V^{> \lambda}_m \to V^{> \lambda}_n$ and hence we have a map on the subquotients $(\iota_{n,m})_{\ast}: V^{\geq \lambda}_m/V^{>\lambda}_m \to V^{\geq \lambda}_n/V^{>\lambda}_n$

Let $V_{\bullet}$ and $W_{\bullet}$ be two finitely generated $\FI$--modules and suppose that, for all $n$, we have a map $r_n : V_n \to W_n$ of $\Sn_n$ representations. Then, by Schur's lemma, the $r_n$ induce linear maps $(V_n)_{\lambda[n]} \to (W_n)_{\lambda[n]}$ and hence, for $n$ sufficiently large, induce maps $V_{\lambda} \to W_{\lambda}$. Our subject in this section is how to compute those maps. We abbreviate $A = V_{\lambda}$ and $B = W_{\lambda}$. Finally, we recall the notation $\alpha_{\lambda[n]}$ for the inclusion $A \otimes \Sp(\lambda[n]) \to V_n$ and $\beta_{\lambda[n]}$ for the surjection $W_n \to B \otimes \Sp(\lambda[n])$.

The key technical lemma of this section is the following.

\begin{lemma}\label{comdi}
With $V_\bullet,W_\bullet,\lambda$ as above, and with $m \leq n \leq q$ in the stable range, all four maps obtained through composition from $A \otimes \Sp(\lambda[m])$ to $B \otimes \Sp(\lambda[q])$ in the following diagram are equal.
\begin{eqnarray}
\begin{CD}
@. @. W_q @>>\beta_{\lambda[q]}> B \otimes \Sp(\lambda[q])\\
@. @. @AA(\iota_{n,q}))_\ast A             @AA \mathrm{Id} \otimes \eta_{n,q}A\\
A \otimes \Sp(\lambda[n]) @>\alpha_{\lambda[n]}>> V_n @>>r_n> W_n @>>\beta_{\lambda[n]}> B \otimes\Sp(\lambda[n])\\
@AA \mathrm{Id} \otimes  \eta_{m,n}A @AA(\iota_{m,n})_\ast A\\
A \otimes \Sp(\lambda[m]) @>\alpha_{\lambda[m]}>> V_m @. @.
\end{CD}\label{cd}
\end{eqnarray}
\end{lemma}

\begin{remark}
The two squares in the diagram (\ref{cd}) need \emph{not} be  commutative.
\end{remark}

\begin{proof}
The two maps arising from the bottom left square are equal modulo $V^{> \lambda}_n$. On the other hand, $r_n$ maps $V^{> \lambda}_n$ to $W^{> \lambda}_n$, which is annihilated by $\beta_{\lambda[n]}$, and is mapped to $W^{> \lambda}_q$ by $(\iota_{n,q})_\as$. All of these facts imply that the choice of map from $A \otimes \Sp(\lambda[m])$ does not effect the overall composition.

Similarly, the two maps arising from the upper right square agree when restricted to $W^{\geq \lambda}_n$. Starting from $A \otimes \Sp(\lambda[m])$, all choices of maps land in this subspace.
\end{proof}

\begin{definition}
Let $V,W,\lambda$ be as above, and let $m \leq n \leq q$ be in the stable range. Then we write $\delta_{m,n,q}: A \otimes \Sp(\lambda[m]) \rightarrow B \otimes \Sp(\lambda[q])$ to denote the equal maps of Lemma \ref{comdi}.
\end{definition}

Our next goal will be to relate the map $\delta_{m,n,q}$ to $r_{n,\lambda}$. Composition in (\ref{comdi}) along the path up-right-right-right-up yields the equality
\[
\delta_{m,n,q} = r_{n,\lambda} \otimes \eta_{m,q}.
\]
On the other hand, composition along the path right-up-right-up-right gives:
\[
\delta_{m,n,q} = \beta_{q,\lambda} \circ (\iota_{n,q})_\ast \circ r_n \circ (\iota_{m,n})_\ast \circ \alpha_{m,\lambda}.
\]
So we have:

\begin{eqnarray}
r_{n,\lambda} \otimes \eta_{m,q} = \beta_{q,\lambda} \circ (\iota_{n,q})_\ast \circ r_n \circ (\iota_{m,n})_\ast \circ \alpha_{m,\lambda}.\label{key}
\end{eqnarray}

We will find that in the case of $\FI$--sets, the right hand side of the above equality is straight forward to compute. This will allow us to give an explicit description of $r_{n,\lambda}$.

Once and for all, fix some $m$ in the stable range along with a vector $x \in \Sp(\lambda[m])$. For any $n \geq m$, choose some integer $q$ such that $m \leq n \leq q$, and pick a linear functional $\psi: \Sp(\lambda[q]) \rightarrow \Q$ such that $\psi((\iota_{n,q})_\ast(x)) = 1$. We further impose the requirement that $\psi$ is equivariant with respect to the action of $\Sn_{q-m}$, thought of as the automorphism group of the set $\{m+1,\ldots,q\}$. Note that this can be done via an averaging trick, because $\Sn_{q-m}$ acts trivially on the image of $\iota_{m,q}$.

Let $\{b_j\}$ be any fixed basis of $B$ and $\{a_i\}$ a fixed basis for $A$. If we write $\{b_j^{\vee}\}$ to denote the dual basis of $\{b_i\}$, then we find that the $(i,j)$-th entry of $r_{n,\lambda}$ with respect to these bases is

\begin{eqnarray}
\langle b_j^\vee \otimes \psi,\ \delta(a_i \otimes x) \rangle= \langle b_j^\vee \otimes \psi,\ \beta_{q,\lambda} \circ (\iota_{n,q})_\ast \circ r_n \circ (\iota_{m,n})_\ast \circ \alpha_{m,\lambda}(a_i \otimes x) \rangle \label{key2}
\end{eqnarray}

Our goal in the next section will be to specialize this setup to $\FI$--modules which arise from linearizations of $\FI$--sets. We will see that in this setting the right hand side of (\ref{key}) is computable enough for us to conclude Theorem \ref{relationthm}.

\subsection{The proof of Theorem \ref{relationthm}}\label{relproof2}

In this section we specialize the setup in the previous section to $\FI$--modules arising from linearizations of $\FI$--sets. Let $X_\dt$ and $Y_\dt$ be  finitely generated $\FI$--sets, and let $R_\dt$ be a relation between these sets. We will write $V_\dt = k X_\dt$, $W_\dt = k Y_\dt$ and we write $r_n$ for the map $V_n \to W_n$ induced by the relation $R_n$. For $x \in X_n$ and $y \in Y_n$ we will write $x \sim y$ to indicate $(x,y) \in R_n$.

Our first reduction will be to limit the total number of stable orbits of our $\FI$--sets. The stable orbits of the relation $R_\dt$ are subsets of products of stable orbits, one from $X_\dt$ and one from $Y_\dt$. It follows that the map $r_n$ will split along such products. Thus, it suffices to understand the map $r_n$ restricted to a chosen pair of orbits. We therefore may  and do assume that $X_\dt$ and $Y_\dt$ have a unique stable orbit.

With this assumption in mind, Theorem \ref{structurethm} implies that for $n \gg 0$,
\[
X_n = M(C)_n, \hspace{.5cm} Y_n = M(D)_n
\]
where $C$ is an $\Sn_r$-set for some $r$ and similarly $D$ is an $\Sn_t$-set for some $t$.
 In particular, we may think of $V_n$ as having a basis of pairs $(K,c)$, where $K$ is an $r$-element subset of $[n]$, and $c \in C$ (see Definition \ref{inducedset} for how the action is defined on this basis). A similar description exists for $W_n$, which will have a basis of pairs $(T,d)$.

\textbf{For the remainder of this section, it will go without saying that $K$ denotes a set of size $r$ and $T$ denotes a subset of size $t$.}\\

We may write $\alpha_{m,\lambda}(a_i \otimes x) = \sum_{(K,c)} \gamma_{K,c} (K,c)$, for some scalars $\gamma_{K,c}$, where $K \subseteq [m]$. Our job will be to compute
\[
\langle b_j^\vee \otimes \psi,\ \beta_{q,\lambda} \circ (\iota_{n,q})_\ast \circ r_n \circ (\iota_{m,n})_\ast(K,c) \rangle .
\]
By definition we have $(\iota_{m,n})_\ast(K,c) = (K,c)$, where $K$ is thought of as a subset of $[n]$, and 
\[
r_n(K,c) = \sum_{(K,c) \sim (T,d), \ T \subset [n]}(T,d).
\]
Thus, we have reduced the problem to needing to compute
\begin{eqnarray}
\left\langle b_j^\vee \otimes \psi,\ \beta_{q,\lambda}{\Big(} \sum_{\substack{ (K,c) \sim (T,d) \\ T \subset [n] \subseteq [q]}} (T,d) {\Big)} \right\rangle = 
\sum_{\substack{ (K,c) \sim (T,d) \\ T \subset [n] \subseteq [q]}} \left\langle b_j^\vee \otimes \psi,\ \beta_{q,\lambda}{\Big(}  (T,d) {\Big)} \right\rangle
.\label{key3}
\end{eqnarray}
We observe that $\psi$ was constructed to be $\Sn_{q-m}$--equivariant, and that $\beta_{q,\lambda}$ is $\Sn_{q}$--equivariant. 
This implies that the summand on the right hand side of~\eqref{key3} is unchanged by the action of $\Sn_{q-m}$ on pairs $(T,d)$. In particular, we may gather together those terms in the sum
according to $S = T \cap [m]$. This yields the expression
\[
\sum_{\substack{S \subseteq [m]\\ d \in D}} \phi_{S,d} \ \#\{T \subseteq [n] \mid (K,c) \sim (T,d) \text{ and } T \cap [m] = S\}
\]
where $\phi_{S,d}$ is some constant. We conclude our proof with the following lemma.

\begin{lemma}
Using the notation of this section, the quantity $|\{T \mid (K,c) \sim (T,d) \text{ and } T \cap [m] = S\}|$ is either equal to 0 for all $n \geq m$ or to $\binom{n-m}{t - |S|}$ for all $n \geq m$. \\
\end{lemma}

\begin{proof}
The key observation which will allow us to prove the lemma is the following. If there is some $(T,d)$ with $(K,c) \sim (T,d)$ and $T \cap [m] = S$, then \emph{every} choice of $T' \subseteq [n]$ with $T' \cap [m] = S$ and $|T'| = t$ has $(K,c) \sim (T,d)$. Indeed, one may find a permutation $\sigma \in \Sn_{n-m}$ which maps $T$ to $T'$, and has $\sigma(T,d) = (T',d)$. Moreover, because $\sigma$ fixed $[m]$ we must have $\sigma(K,c) = (K,c)$. Thus, because our relation is equivariant under the action of the symmetric group, we conclude that $(K,c) \sim (T',d)$. This implies that the set in question either has size zero or $\binom{n-m}{t-|S|}$.

To conclude the proof, we must show that for $n \gg 0$ there exists some set $T' \subseteq [n+1]$ with $|T'| = t, (K,c) \sim (T',d),$ and $T' \cap [m] = S$ if and only if there exists some $T \subseteq [n]$ with the same properties. Indeed, this follows from the fact that the relation $R_\dt$ is finitely generated and we have taken $n$ to be in the stable range.
\end{proof}

\subsection{Applications of Theorem \ref{relationthm}}
\label{mainapp}
In this section, we consider applications of Theorem \ref{relationthm}. In particular, we prove Corollary \ref{maincor}, and apply it to various cases.

\begin{proof}[The proof of Corollary \ref{maincor}]
Let $X_\dt$ be an element-stable $\FI$--set, and let $R_\dt$ be a self relation. For any partition $\lambda$ we write $r_{n,\lambda}$ for the restriction of $r_n$ to the $\lambda[n]$--isotypic piece of $\Q X_n$. Then, by Theorem \ref{relationthm}, there exists a choice of bases such that for all $n \gg 0$ the maps $r_{n,\lambda}$ take the form,
\[
\begin{pmatrix}
A_\lambda(n)& 0& 0 &\ldots & 0& 0\\
0 & A_\lambda(n) & 0 &\ldots & 0 &0\\
\vdots\\
0 & 0 & 0 &\ldots & 0 & A_{\lambda}(n)
\end{pmatrix}
\]
where $A_{\lambda}(n)$ is a square matrix of fixed (non-varying in $n$) dimension with entries in $\Q[n]$, and the total number of blocks is precisely $\dim_\Q \Sp(\lambda[n])$. Moreover,  representation stability theory \cite{CEF} implies that the total number of non-zero $r_{n,\lambda}$ is a constant independent of $n$. Therefore, to understand the eigenvalues of $r_{n}$ it suffices to understand the eigenvalues of $A_\lambda(n)$.

We may think of $A_\lambda(n)$ as being a matrix over the field $\Q(n)$. With this perspective, it becomes clear that we may factorize the characteristic polynomial of $A_\lambda(n)$, over some algebraic extension of $\Q(n)$, as
\[
\chi_{\lambda}(n,x) = \prod_i (x-f_i(n))^{e_i}
\]
where $e_i \geq 1$ are some integers, and $f_i(n)$ is some function which is algebraic over $\Q(n)$. This allows us to deduce that the eigenvalues of $r_n$ are algebraic functions over $\Q(n)$ (in fact, they are integral over $\Q[n]$), as desired.

We next must argue that the functions $f_i(n)$, with $i$ varying, only agree for finitely many values of $n$. In other words, if $f(n)$ and $g(n)$ are distinct algebra functions, we must argue that $f(n)=g(n)$ for only finitely many $n$. Indeed, let $P(z,n)$ be the minimal polynomial of minimal $z$ degree with $P(f(n), n) = P(g(n), n) = 0$. 
By the minimality of the degree of $P$, the polynomial $P$ is squarefree as a polynomial in $z$, so the discriminant $\Delta(n)$ of $P$ with respect to $z$ is a nonzero polynomial in $n$. For any $n$ which is not a root of $\Delta(n)$, the roots of $P(z,n)=0$ are distinct, so $f(n) \neq g(n)$ for such an $n$. We note for future use that we have proved that, even if we allow $n$ to take real, non-integer values, there are only finitely many $n$ for which $f(n) = g(n)$.


To conclude, we must show that the multiplicities of these distinct eigenvalues are equal to polynomials in $n$. This follows from the fact that the eigenvalues of each $A_{\lambda}(n)$ have constant multiplicity, while the total number of $A_{\lambda}(n)$ which appear in the above matrix is precisely $\dim_\Q S(\lambda)_n$, a polynomial in $n$.
\end{proof}

Once again calling upon the example of $\FI$--graphs with the edge relation, we see that Corollary \ref{maincor} implies Theorem H of \cite{RW}. Recall that for a graph $G$ the Laplacian of $G$ is the matrix $D-A_n$, where $D$ is the diagonal matrix of degrees of vertices of $G$, and $A_n$ is the adjacency matrix.\\

\begin{theorem}[Theorem H of \cite{RW}]\label{rwthm}
Let $G_\dt$ be a vertex-stable $\FI$--graph, and let $r_n$ denote either the adjacency matrix or Laplacian of $G_n$. We may write the distinct eigenvalues of $r_n$ as,
\[
\lambda_1(n) < \lambda_2(n) < \ldots < \lambda_{l(n)}(n),
\]
for some function $l(n)$. Then for all $n \gg 0$
\begin{enumerate}
\item $l(n) = l$ is constant. In particular, the number of distinct eigenvalues of $r_n$ is eventually independent of $n$;
\item for any $i$ the function
\[
n \mapsto \lambda_i(n)
\]
agrees with a function which is algebraic over $\Q(n)$.
\item for any $i$ the function
\[
n \mapsto \text{ the multiplicity of $\lambda_i(n)$}
\]
agrees with a polynomial.
\end{enumerate}
\end{theorem}

\begin{proof}
To make sense of the theorem statement, note that in this case $r_n$ is symmetric, whence our matrix must have real eigenvalues. In particular, the functions $f_i(n)$ from the  proof of Corollary \ref{maincor} must be real-valued. Because we know that the $f_i(n)$ evaluate to distinct real numbers for $n \gg 0$, it follows that we may order them using the usual order on $\R$. All of this put together imply Theorem H of \cite{RW} for the adjacency matrix.

To prove this statement for the Laplacian matrix, we note that it is shown in \cite{RW} that finitely generated $\FI$--graphs have vertex degrees which agree with polynomials in $n$ for $n \gg 0$. Therefore, the Laplacian can be expressed as a $\Q[n]$--linear combination of relations. It follows that the Laplacian will satisfy the conclusion of Theorem \ref{relationthm}, and therefore will also satisfy the conclusions of Corollary \ref{maincor}
\end{proof}

\end{document}